\documentclass[a4paper,12pt]{article}

\usepackage[left=2cm,top=2.5cm,right=2cm,bottom=2.5cm]{geometry}

\usepackage[usenames,dvipsnames]{color}
\usepackage{amsthm}
\usepackage{amsmath,amssymb}
\usepackage{enumitem}

\newtheorem{theo}{Theorem}[section] 
\newtheorem*{theo*}{Theorem}
\newtheorem{lem}[theo]{Lemma} 
\newtheorem{prop}[theo]{Proposition}

\newtheorem{rem}{Remark}[section]
\newtheorem{definition}{Definition}[section]

\newcommand{\R}{\mathbb{R}}\newcommand{\Rn}{\R^n}\newcommand{\Rnn}{\R^{n\times n}}
\newcommand{\N}{\mathbb{N}}
\DeclareMathOperator{\supp}{supp}
\DeclareMathOperator{\diver}{div}

\DeclareMathOperator{\tr}{tr} 

\DeclareMathOperator{\dist}{dist}
\DeclareMathOperator{\loc}{loc}
\DeclareMathOperator{\sym}{sym}

\DeclareMathOperator{\id}{id}

\renewcommand{\O}{\Omega}
\renewcommand{\a}{\alpha}

\renewcommand{\d}{\delta}

\newcommand{\G}{\Gamma}
\newcommand{\p}{\partial}

\newcommand{\e}{\varepsilon}

\usepackage{amsmath}
\usepackage{latexsym}
\usepackage{amssymb}

\usepackage{url}

\def\XXint#1#2#3{{\setbox0=\hbox{$#1{#2#3}{\int}$}
		\vcenter{\hbox{$#2#3$}}\kern-.5\wd0}}

\title{Eringen's model via linearization of nonlocal hyperelasticity}

\author{J. C. Bellido \\
{\small E.T.S.I.\ Industriales, Department of Mathematics, Universidad de Castilla-La Mancha,} \\
{\small 13071-Ciudad Real, Spain. Email \url{JoseCarlos.Bellido@uclm.es}},\\
J. Cueto \\
{\small  Department of Mathematics, University of Nebraska-Lincoln,} \\
{\small Lincoln NE 68588-0130, US. Email \url{jcuetogarcia2@unl.edu}}, \\
C. Mora-Corral \\
{\small Departamento de Matem\'aticas, Universidad Aut\'onoma de Madrid,} \\
{\small 28049 Madrid, Spain and Instituto de Ciencias Matem\'aticas, } \\
{\small CSIC-UAM-UC3M-UCM, 28049 Madrid, Spain. Email \url{Carlos.Mora@uam.es}}}

\date{\today}

\usepackage{graphicx}
\usepackage{wrapfig}

\begin{document}

\maketitle

\pagestyle{empty}



\pagestyle{plain}

\begin{abstract}
We consider Riesz' fractional gradient and a truncated version of it.
The equations of nonlocal nonlinear elasticity based on those gradients are known.
We perform a formal linearization and arrive at the equations of linear elasticity based on those nonlocal operators.
We prove the existence of solutions of the linear equations, notably, by a nonlocal version of Korn's inequality.
Finally, we show that the linearizations obtained are particular cases of Eringen's model with singular kernels.
\end{abstract}

\section{Introduction}

In nonlinear elastostatics, a fundamental question is the existence of equilibrium solutions of the equations of nonlinear elasticity, which often arise as minimizers of the elastic energy
\begin{equation*}
 \int_{\O} W (x, D u (x)) \, dx
\end{equation*}
of a deformation $u : \O \to \Rn$.
Here $\O$ is an open bounded subset of $\Rn$ representing the reference configuration of the body (where $n=3$ is the physically relevant case), and $W : \O \times \Rnn \to \R \cup \{\infty\}$ is the elastic stored-energy function of the material.
The usual approach for finding such minimizers is the direct method of the calculus of variations.
This theory is well established since the pioneering paper of Ball \cite{Ball77} and its many subsequent refinements. 

On the other hand, nonlocal models in solid mechanics have experienced a huge development in the last two decades, especially from the introduction of the peridynamics model by Silling \cite{Silling2000}.
Many refinements have been introduced since then and, particularly, nonlocal models based on a nonlocal gradient have received a great attention as an adequate substitute of local models \cite{DeGuMeSc21, DeGuOlKa21}.

In general, a nonlocal gradient of a function $u : \O \to \R$ takes the form 
\[
 \mathcal{G}_\rho u(x)= \int_\O \frac{u(x)-u(y)}{|x-y|} \frac{x-y}{|x-y|} \rho(x-y)\, dy ,
\]
for a suitable kernel $\rho$, usually with a singularity at the origin.
The choice of $\rho$ determines the nonlocal gradient, which, in turn, specifies the functional space.

The most popular nonlocal gradient is possibly Riesz' $s$-fractional gradient, which is denoted by $D^s u$ and corresponds to the choices $\O = \Rn$ and $\rho(x) = \frac{c_{n,s}}{|x|^{n-1+s}}$ for some constant $c_{n,s}$; see \cite{ShSp15,ShSp17}.
Here $0<s<1$ is the degree of differentiability.
While Riesz' fractional gradient enjoys many desirable properties (it is invariant under rotations and translations, it is $s$-homogeneous under dilations; see \cite{Silhavy20}), it has the drawback that the integral defining it is over the whole space, which makes it unsuitable for solid mechanics where the body is represented by a bounded domain $\O \subset \Rn$.
An adaptation of Riesz' $s$-fractional gradient for bounded domains was done by the authors in \cite{BeCuMo22}.
Precisely, for a $C^{\infty}_c$ function $u$, its nonlocal gradient is defined as
\[
 D_\delta^s u(x) = c_{n,s} \int_{B(x,\delta)} \frac{u(x)-u(y)}{|x-y|}\frac{x-y}{|x-y|}\frac{w_\d(x-y)}{|x-y|^{n-1+s}} \, dy ,
\]
where $w_{\d}$ is a fixed function in $C^{\infty}_c (B(0, \d))$ satisfying some natural properties to be a truly cut-off function.
Here $\d>0$ plays the role of the horizon (in the terminology of peridynamics), i.e., the maximum interaction distance.

The existence of minimizers of functionals 
\begin{equation}\label{eq:integralDsd}
 \int_{\O} W (x, D^s_{\d} u(x)) \, dx
\end{equation}
was done in \cite{BeCuMo23} under the assumption of polyconvexity of $W$ (where we also wrote down the corresponding Euler--Lagrange equations), and in \cite{CuKrSc22} under the assumption of quasiconvexity.

Coming back to nonlocal theories in linear elasticity, one of the most popular, particularly in the engineering community, is that due to Eringen \cite{EringenBook}. Eringen's model in linear elasticity is similar to the classical Lam\'e-Navier system but differing from it in the fact that the classical stress tensor is substituted by a nonlocal one, computed from the local stress tensor through averaging:
\[
\sigma (x) = \int_\O A(x,x') C D_{\sym} v(x') \, dx' ,
\]
where $\sigma$ is the averaged stress, $A$ the nonlocal interaction kernel, $C$ the fourth-order elastic tensor and $D_{\sym}v$ the linearized strain (the symmetric part of the displacement gradient). Different variational formulations for this model are given in \cite{Polizzotto}. Thus, this is an integral theory, as the stress at a point in the domain depends, through averaging, on the stress at points around it.
Because of this fact, in the engineering literature it is said that the model belongs to the framework of strong nonlocal theories.
This model has been widely used in a variety of applications in mechanics, and nowadays it is attracting a revitalized interest owing to its applicability to the modeling of nanobeams and nanobars \cite{Romano}. A study on existence and uniqueness of solutions of Erigen's model of linear elasticity can be found in \cite{EvBe}. 

In this article we link the two theories above in the following way.
We start with the Euler--Lagrange equations associated to the functional \eqref{eq:integralDsd}.
They can be interpreted as the equilibrium equations in nonlocal elasticity.
Then we perform a linearization of those equations, which is merely formal, as done classically in texts of elasticity.
As expected, the linear equations are remarkably similar to the linear ones, just by replacing the local differential operators (notably, gradient, divergence and Laplacian) by their nonlocal counterparts.
In a further stage, we prove existence and uniqueness of those linear equations under the same assumptions as in the classical case, namely, the positive definiteness of the elasticity tensor.
This tensor is a local quantity acting on nonlocal gradients.
The key ingredient for the well-posedness of the linear equations is a suitable version of Korn's inequality for nonlocal gradients.
Our proof of that inequality is based on the classical Korn inequality together with a procedure described in \cite{BeCuMo23} and \cite{CuKrSc22} (and, earlier, in \cite{KrSc22} for the case of the Riesz potential) to translate results from the local case to the nonlocal context.

For completeness, we do a parallel analysis in the fractional case, i.e., with Riesz' fractional gradient.
The caveat of this process in the fractional case is that the linearization is performed around the identity (as in classical elasticity), but the identity has infinite energy, i.e., it does not belong to the associated Bessel space because that space involves the whole $\Rn$.
Hence, the linearization process is dubious, and that is why we prefer to present first the nonlocal case, which is done over a bounded domain and where the identity belongs to the relevant space.
 
These two analyses together allow for a better comparison with Eringen's model.
In general terms, we will see that two suitable choices of the kernel in Eringen's model are equivalent, respectively, to the nonlocal and fractional linear models derived earlier. More concretely we show in Section 8, that, both functional spaces where problemas are set and bilinear forms defining those problems, coincide for the linealization we obtain both for nonlocal and fractional problems with Eringen's problems with nonlocal kernels related to function $Q_\d^s$ (defined in the next section) and to Riesz potential respectively. 

We mention two articles that also deal with the equations of linear elasticity for fractional operators.
First, \cite{Silhavy22} introduced the linear isotropic model in fractional elasticity, but without a derivation: just by subtituting the local operators by nonlocal ones.
His model coincides with ours in the fractional case, and, in addition, he proves a version of Korn's inequality, with techniques relying on Fourier and Riesz transforms.
Second, the article \cite{Scott23} also studies the linear isotropic fractional elasticity equation, with a motivation coming from state-based peridynamics and the fractional powers of the local operator.

The outline of this article is as follows.
In Section \ref{se:notation} we establish some general notation.
Section \ref{se:functional} describes the nonlocal framework of the article: nonlocal gradient and divergence, functional space with useful embeddings, and some differential identities.
In Section \ref{eq:derivativelinear} we do an elementary calculation showing that the nonlocal derivative of an affine map is the associated linear map.
Section \ref{se:linearization} performs the linearization of the equations of nonlocal elasticity based on the nonlocal gradient $D^s_{\d}$.
In Section \ref{se:existence} we prove the relevant version of Korn's inequality to show existence and uniqueness of solutions of these equations.
The analogous results based on the fractional gradient $D^s$ are shown in Section \ref{se:fractional}: linearization, existence and uniqueness.
In the final Section \ref{se:Eringen} we explain Eringen's model and how it is connected with the linearizations found earlier.

\section{Notation}\label{se:notation}

We fix the dimension $n \in \N$ of the space, an open bounded set $\O$ of $\Rn$ representing the body, the degree $0< s < 1$ of differentiability, a $\d>0$ indicating the horizon (the interaction distance between the particles of the body), and an exponent $1 \leq p < \infty$ of integrability.
The dual exponent of $p$ is $p' = \frac{p}{p-1}$.

The notation for Sobolev $W^{1,p}$ and Lebesgue $L^p$ spaces is standard, as is that of smooth functions of compact support $C^{\infty}_c$.
The support of a function is denoted by $\supp$.
We will indicate the domain of the functions, as in $L^p (\O)$; the target is indicated only if it is not $\R$.
When using the norm in those spaces, the target is omitted, as in $\left\| \cdot \right\|_{L^p (\O)}$.

We write $B(x,r)$ for the open ball centred at $x \in \Rn$ of radius $r >0$.
The complement of an $A \subset \Rn$ is denoted by $A^c$.

We denote by $\Rnn$ the set of $n \times n$ matrices.
The identity matrix is $I$ and the identity map is $\id$.
The subset of  symmetric matrices is $\Rnn_{\sym}$.
The symmetric part $H_{\sym}$ of an $H \in \Rnn$ is
\[
 H_{\sym} = \frac{H + H^T}{2} ,
\]
where $H^T$ is the transpose of $H$.
Similarly, $D_{\sym}$, $D^s_{\sym}$ and $D^s_{\d, \sym}$ denote the symmetric parts of the gradients $D$, $D^s$ and $D^s_{\d}$.
The inner product in $\Rn$ is denoted by $\cdot$, while $:$ is the inner product in $\Rnn$.
The tensor product of two vectors in $\Rn$ is denoted by $\otimes$.

The divergence of a matrix is the column vector whose components are the divergence of the rows of the matrix.

The operation of convolution is denoted by $*$.

\section{Functional space}\label{se:functional}

This section is a compendium of the definitions and results taken from \cite{BeCuMo22,BeCuMo23} on the nonlocal gradient and divergence, as well as their associated space $H^{s,p,\d} (\O)$.

Apart from $\O$, the sets $\O_{\d}= \O + B(0,\d)$ and $\O_{-\d} := \{ x \in \O : \dist (x, \O^c) > \d \}$ will also be relevant. 
The number $\d>0$ is chosen small enough so that $\O_{-\d}$ is not empty.

Let $w_\d: \Rn\to [0,+\infty)$ be a cut-off function, and  $\rho_\d: \Rn\to [0,+\infty)$ defined as 
\[\rho_\delta(x)=\frac{1}{\gamma (1-s)|x|^{n-1+s}}w_{\d}(x),\] 
where the constant $\gamma(s)$ is given by
\begin{equation}\label{eq:gamma}
	\gamma(s)=\frac{\pi^{\frac{n}{2}}2^s\Gamma\left(\frac{s}{2}\right)}{\Gamma\left(\frac{n-s}{2}\right)}
\end{equation}
and $\Gamma$ is Euler's gamma function.
Moreover, set
	\begin{equation*}
		c_{n,s}:= \frac{n-1+s}{\gamma(1-s)}.
	\end{equation*}
The precise assumptions on $w_{\d}$ are as follows:
\begin{enumerate}[label=\alph*)]
	\item\label{item:wA} $w_\d$ is radial and nonnegative.
	\item $w_\d \in C_c^\infty(B(0,\delta))$.
	\item There are constants $a_0>0$ and $0<b_0<1$ such that $0\leq w_\d \leq a_0$, and $w_\d |_{B(0, b_0\delta)} =a_0$.
	\item\label{item:wdecreasing} $w_\d$ is radially decreasing, i.e. 
	$w_\d (|x_1|) \geq w_\d (|x_2|)$ if $|x_1| \leq |x_2|$.
	
		\item\label{item:wnormalization} $\displaystyle \int_{B(0,\d)}\frac{w_\d(z)}{|z|^{n+s-1}} \, dz = \frac{n}{c_{n,s}}$.

\end{enumerate}

Conditions \ref{item:wA}--\ref{item:wdecreasing} are taken from \cite[Sect.\ 3]{BeCuMo22}, where it is remarked that condition \ref{item:wdecreasing} can be weakened.
Condition \ref{item:wnormalization} is a normalization.
In \cite{BeCuMo22} we did not assume any particular normalization; in fact, there are several natural options, and we have chosen \ref{item:wnormalization} in view of the result of Section \ref{eq:derivativelinear}: the nonlocal derivative of a linear map is the matrix representing that linear map.
Note that \ref{item:wnormalization} can be equivalently written as $\|\rho_\d\|_{L^1(\Rn)} = \frac{n}{n-1+s}$.  

The definitions of the nonlocal gradient and divergence for smooth functions are as follows.
\begin{definition} \label{def: nonlocal gradient}
	\begin{enumerate}[label=\alph*)]
		\item \label{item:Dsdu}
The nonlocal gradient of $u\in C_c^{\infty} (\Rn)$ is defined as
		\begin{equation*}
			D_\delta^s u(x)= c_{n,s} \int_{B(x,\delta)} \frac{u(x)-u(y)}{|x-y|}\frac{x-y}{|x-y|}\frac{w_\d(x-y)}{|x-y|^{n-1+s}} \, dy , \qquad x \in \Rn ,
		\end{equation*}
and of $u\in C_c^{\infty} (\Rn, \Rn)$ as
\begin{equation*}
 D_\delta^s u(x)= c_{n,s} \int_{B(x,\delta)} \frac{u(x)-u(y)}{|x-y|} \otimes \frac{x-y}{|x-y|}\frac{w_\d(x-y)}{|x-y|^{n-1+s}} \, dy , \qquad x \in \Rn .
\end{equation*}

		\item
The nonlocal divergence of $u \in C^{\infty}_c (\Rn,\Rn)$ is defined as
\begin{equation*}
			\diver_{\delta}^s u(x)=  c_{n,s} \int_{B(x,\delta)} \frac{u(x)-u(y)}{|x-y|}\cdot\frac{x-y}{|x-y|}\frac{w_\d(x-y)}{|x-y|^{n-1+s}} \, dy, \qquad x \in \Rn ,
		\end{equation*}
and of $u\in C_c^{\infty} (\Rn, \Rnn)$ as
\[
\diver_{\delta}^s u(x)=  c_{n,s} \int_{B(x,\delta)} \frac{u(x)-u(y)}{|x-y|} \frac{x-y}{|x-y|}\frac{w_\d(x-y)}{|x-y|^{n-1+s}} \, dy, \qquad x \in \Rn .
\]
	\end{enumerate}
\end{definition}



Notice that the integrals in Definition \ref{def: nonlocal gradient} are absolutely convergent because $u$ is Lipschitz and $\rho_{\d} \in L^1 (\Rn)$.

We extend Definition \ref{def: nonlocal gradient}\,\ref{item:Dsdu} to a broader class of functions.

\begin{definition}\label{def: nonlocal gradient 2}
	\begin{enumerate}[label=\alph*)]
		\item
		Let $u \in L^1 (\O_\d)$ be such that there exists a sequence of $\{ u_j \}_{j \in \N} \subset C^{\infty}_c (\Rn)$ converging to $u$ in $L^1 (\O_\d)$ and for which $\{ D_\d^s u_j \}_{j \in \N}$ converges to some $U$ in $L^1 (\O, \Rn)$.
		We define $D_\d^s u$ as $U$.
		
		\item
		Let $\phi \in L^1 (\O_\d, \Rn)$ be such that there exists a sequence of $\{ \phi_j \}_{j \in \N} \subset C^{\infty}_c (\Rn, \Rn)$ converging to $\phi$ in $L^1 (\O_\d, \Rn)$ and for which $\{ \diver_\d^s \phi_j \}_{j \in \N}$ converges to some $\Phi$ in $L^1 (\O)$.
		We define $\diver_\d^s \phi$ as $\Phi$.
	\end{enumerate}
\end{definition}

It was shown in \cite[Lemma 3.3]{BeCuMo22} that the above definitions of $D_\d^s u$ and $\diver_\d^s \phi$ are independent of the sequence chosen, and that the assumption $L^1$ can be weakened to $L^1_{\loc}$.

An important result is the nonlocal integration by parts formula, which will be used for passing from strong to weak nonlocal formulations. Here, we state a simplified version of \cite[Th.\ 3.2]{BeCuMo22} suitable for our aim in this paper. 

\begin{prop}\label{th:Nl parts}
Suppose that $u \in C_c^{\infty} (\O)$ and $\phi \in C^1_c (\O, \Rn)$.
Then $D^s_{\d} u \in L^{\infty} (\Rn, \Rn)$ and $\diver^s_{\d} \phi \in L^{\infty} (\Rn)$.
Moreover,
		\[\int_{\O} D_\delta^s u(x) \cdot \phi(x) \, dx = - \int_{\O} u(x) \diver_\delta^s \phi (x) \, dx ,\]
and these two integrals are absolutely convergent.
\end{prop}

The functional space associated to these nonlocal differential operators is as follows \cite{BeCuMo22}.

\begin{definition}\label{de:Hspd}
Let $1 \leq p < \infty$.
We define the space $H^{s,p,\d}(\O)$ as the closure of $C^{\infty}_c(\Rn)$ under the norm 
\[
 \left\| u \right\| _{H^{s,p,\d}(\O)} = \left( \left\| u \right\|_{L^p (\O_{\d})}^p + \left\| D_\delta^s u \right\|_{L^p (\O)}^p \right)^{\frac{1}{p}} .
\]
\end{definition}
Thus, functions in $H^{s,p,\d}(\O)$ are defined a.e.\ in $\O_{\d}$, while its gradient (Definition \ref{def: nonlocal gradient 2}) is defined a.e.\ in $\O$.
The space  $H^{s,p,\d}(\O, \Rn)$ is defined analogously, as are the gradient of vector-valued functions and the divergence of matrix-valued functions (Definition \ref{def: nonlocal gradient}).

The space $H^{s,p,\d}(\O)$ satisfies reflexivity and separability properties.  
See \cite[Prop 3.4]{BeCuMo22}.
\begin{prop}
Let $1\leq p < \infty$.
Then $H^{s,p,\d}(\Omega)$ is a separable Banach space.
If, in addition, $p > 1$, it is reflexive.
When $p=2$, it is a Hilbert space under the inner product
\[
 (u, v) \mapsto \int_{\O_{\d}} u\, v \, dx + \int_{\O}  D^s_{\d} u \cdot D^s_{\d} v \, dx .
\]
\end{prop}


The following two results present each natural differential equalities. 

\begin{lem}\label{le:trD}
For all $u \in H^{s,p,\d} (\O, \Rn)$,
\[
 \tr D_\d^s u =\tr (D_\d^s u)^T =\tr D_{\d,\sym}^s u = \diver_\d^s u .
\]
\end{lem}
\begin{proof}
The proof of $\tr D_\d^s u = \diver_\d^s u$ was done in \cite[Lemma 3.4]{BeCuMo23}.
Since $\tr A = \tr A^T = \tr A_{\sym}$ for any $A \in \Rnn$, the result follows. 
\end{proof}

\begin{lem}\label{le:divDT}
For all $u \in C^{\infty}_c (\Rn, \Rn)$,
\[
 \diver^s_{\d} (D^s_{\d} u)^T = D^s_{\d} (\diver^s_{\d} u) .
\]
\end{lem}

\begin{proof}
Due to Definition \ref{def: nonlocal gradient} and a change of variables we have that for all $x \in \Rn$,
\begin{align*}
 & \diver^s_{\d} (D^s_{\d} u)^T (x) = c_{n,s} \int_{B(0,\d)} \frac{(D^s_{\d} u)^T (x) - (D^s_{\d} u)^T (x-z)}{|z|^{n+s}} \frac{z}{|z|} w_{\d} (z) \, dz , \\
 & D^s_{\d} (\diver^s_{\d} u) (x) = c_{n,s} \int_{B(0,\d)} \frac{\diver^s_{\d} u (x) - \diver^s_{\d} u (x-z)}{|z|^{n+s}} \frac{z}{|z|} w_{\d} (z) \, dz .
\end{align*}
For the same reason,
\begin{align*}
 & (D^s_{\d} u)^T (x) - (D^s_{\d} u)^T (x-z) \\
 & = c_{n,s} \int_{B(0,\d)}\frac{\bar{z}}{|\bar{z}|} \otimes \frac{u(x) - u(x- \bar{z}) - [u (x-z) - u (x-z-\bar{z})]}{|\bar{z}|^{n+s}} w_{\d} (\bar{z}) \, d\bar{z}
\end{align*}
and
\begin{align*}
 & \diver^s_{\d} u (x) - \diver^s_{\d} u (x-z) \\
 & = c_{n,s} \int_{B(0,\d)} \frac{u (x) - u (x-\bar{z}) - [u (x-z) - u (x-z-\bar{z})]}{|\bar{z}|^{n+s}} \cdot \frac{\bar{z}}{|\bar{z}|} w_{\d} (\bar{z}) \, d\bar{z} ,
\end{align*}
so in order to prove the desired equality we have to show
\begin{equation}\label{eq:intint}
\begin{split}
 & \int_{B(0,\d)} \int_{B(0,\d)}\frac{\bar{z}}{|\bar{z}|} \otimes \frac{u(x) - u(x- \bar{z}) - [u (x-z) - u (x-z-\bar{z})]}{|\bar{z}|^{n+s}} w_{\d} (\bar{z}) \, d\bar{z} \frac{w_{\d} (z)}{|z|^{n+s}} \frac{z}{|z|} \, dz \\
 & = \int_{B(0,\d)} \int_{B(0,\d)} \frac{u (x) - u (x-\bar{z}) - [u (x-z) - u (x-z-\bar{z})]}{|\bar{z}|^{n+s}} \cdot \frac{\bar{z}}{|\bar{z}|} w_{\d} (\bar{z}) \, d\bar{z} \frac{w_{\d} (z)}{|z|^{n+s}} \frac{z}{|z|} \, dz .
\end{split}
\end{equation}
Observe that both integrals above are absolutely convergent.
Indeed, it is enough to notice that, as $u \in C^{\infty}_c (\Rn, \Rn)$, a standard Taylor expansion at $x$ of order $2$ shows that there exists $C>0$ such that for all $x \in \Rn$ and $z, \bar{z} \in B (0, \d)$,
\[
 \left| u(x) - u(x- \bar{z}) - [u (x-z) - u (x-z-\bar{z})] \right| \leq C \left| z \right| \left| \bar{z} \right| .
\]
Therefore, to show \eqref{eq:intint} it suffices to check that for all $0 < \e < \d$,
\begin{equation}\label{eq:A0ed}
\begin{split}
 & \int_{A(0, \e, \d)} \int_{A(0, \e, \d)} \frac{\bar{z}}{|\bar{z}|} \otimes \frac{u(x) - u(x- \bar{z}) - [u (x-z) - u (x-z-\bar{z})]}{|\bar{z}|^{n+s}} w_{\d} (\bar{z}) \, d\bar{z} \frac{w_{\d} (z)}{|z|^{n+s}} \frac{z}{|z|} \, dz \\
 & = \int_{A(0, \e, \d)} \int_{A(0, \e, \d)} \frac{u (x) - u (x-\bar{z}) - [u (x-z) - u (x-z-\bar{z})]}{|\bar{z}|^{n+s}} \cdot \frac{\bar{z}}{|\bar{z}|} w_{\d} (\bar{z}) \, d\bar{z} \frac{w_{\d} (z)}{|z|^{n+s}} \frac{z}{|z|} \, dz ,
\end{split}
\end{equation}
where $A(0, \e, \d) := B(0,\d) \setminus B(0,\e)$. 
As
\[
 \int_{A(0, \e, \d)} \frac{w_\d(y)}{|y|^{n+s}} \frac{y}{|y|} \, dy = 0 ,
\]
equality \eqref{eq:A0ed} is equivalent to
\begin{equation}\label{eq:A0ed2}
\begin{split}
 & \int_{A(0, \e, \d)} \int_{A(0, \e, \d)} \frac{\bar{z}}{|\bar{z}|} \otimes \frac{u (x-z-\bar{z})}{|\bar{z}|^{n+s}} w_{\d} (\bar{z}) \, d\bar{z} \frac{w_{\d} (z)}{|z|^{n+s}} \frac{z}{|z|} \, dz \\
 & = \int_{A(0, \e, \d)} \int_{A(0, \e, \d)} \frac{u (x-z-\bar{z})}{|\bar{z}|^{n+s}} \cdot \frac{\bar{z}}{|\bar{z}|} w_{\d} (\bar{z}) \, d\bar{z} \frac{w_{\d} (z)}{|z|^{n+s}} \frac{z}{|z|} \, dz .
\end{split}
\end{equation}
Using the identity $(a \otimes b) c = (b \cdot c) a$ for all $a, b, c \in \Rn$, as well as Fubini's theorem for the right hand side, we can see that both terms of \eqref{eq:A0ed2} are equal to
\[
\int_{A(0, \e, \d)} \int_{A(0, \e, \d)} \frac{u (x-z-\bar{z})}{|\bar{z}|^{n+s} |z|^{n+s}} \cdot \frac{z}{|z|} \frac{\bar{z}}{|\bar{z}|} w_{\d} (\bar{z}) w_{\d} (z) \, d\bar{z} \, dz
\]
and the proof is complete.
\end{proof}

In order to describe the boundary condition, we recall the set $\O_{-\d} = \{ x \in \O : \dist (x, \O^c) > \d \}$ and define the subspace $H_0^{s,p,\d}(\O_{-\d})$ as the closure of $C_c^{\infty}(\O_{-\d})$ in $H^{s,p,\d}(\O)$.
It is immediate to check that any $u \in H_0^{s,p,\d}(\O_{-\d})$ satisfies $u=0$ a.e.\ in $\O_{\d} \setminus \O_{-\d}$.
We will also use the affine subspace $H^{s,p,\d}_{\id} (\O_{-\d}) = \id + H^{s,p,\d}_0(\O_{-\d})$.

We present an inequality of the type of Sobolev--Poincar\'e, \cite[Sect.\ 6]{BeCuMo22}.
Given $p > 1$ and $0<s<1$ with $sp<n$ we define $p_s^* := \frac{np}{n-sp}$.

\begin{theo}\label{th:PoincareSobolev delta}
Let $1<p<\infty$ and
\[
\begin{cases}
 q\in \left[1, p_s^* \right] & \text{if } sp<n , \\
 q\in[1, \infty) & \text{if } sp=n , \\
 q\in[1,\infty] & \text{if } sp>n .
\end{cases}
\]
Then there exists $C=C(|\O|,n,p,q,s)>0$ such that for all $u \in H_0^{s,p,\d}(\O_{-\d})$,
\[
 \left\| u \right\|_{L^q (\O_{\d})} \leq C \left\| D_\d^s u \right\|_{L^p(\O)} ,
\]
\end{theo}

%
%

An important tool in the subsequent analysis is given by the following result from \cite[Lemma 4.2]{BeCuMo23} and \cite[Prop.\ 2.13]{CuKrSc22} (see also the earlier results \cite[Lemma 4.2 and Prop.\ 4.3]{BeCuMo22}), which states that every nonlocal gradient is in fact a gradient.

\begin{theo}\label{theo:DsdD} Let $1 \leq p < \infty$. There exists a positive radial function $Q_{\d}^s\in L^1(\Rn)$, with $\supp Q_{\d}^s\subset B(0,\delta)$, and strictly positive in the interior of its support, such that for all $u \in H^{s,p,\d} (\O)$, we have that $Q^s_\d*u \in W^{1,p}(\O)$ and
\[
 D^s_{\d} u = D (Q^s_{\d} * u) \quad \text{in } \O.
\]
If, in addition, $u \in C^{\infty}_c (\O_{-\d}, \Rn)$ then $D^s_{\d} u = Q^s_{\d} * D u$.
\end{theo}

As we will do convolutions of $Q^s_{\d}$ with functions defined in $\O_{\d}$, with a small abuse of notation, given $u : \O_{\d} \to \R$ we will write $Q^s_{\d} * u$ as the function defined in $\O$ by
\[
 Q^s_{\d} * u (x) = \int_{B(x,\d)} Q^s_{\d} (x-y) u (y) \, dy ,
\]
whenever the integral is well defined.

\section{Derivative of an affine map}\label{eq:derivativelinear}

In this section we show that the nonlocal derivative of an affine map is the associated matrix of the linear map.

\begin{prop}
Let $F \in \Rnn$ and $a \in \Rn$.
Define $u : \O_\d \to \Rn$ as $u (x) = F x + a$.
Then $u \in H^{s,p} (\O, \Rn)$ and $D^s_{\d} u (x) = F$ for all $x \in \O$.
\end{prop}
\begin{proof}
In order to calculate $D^s_{\d} u$ using Definition \ref{def: nonlocal gradient}\ref{item:Dsdu} we notice that there exists a $C^{\infty}_c(\Rn, \Rn)$ extension of $u$, which will still be called $u$.
A change of variables shows that, for all $x \in \O$,
\[
  D_\d^s u (x) = c_{n,s} \int_{B(x,\d)}\frac{F(x-y)}{|x-y|^{n+s}}\otimes \frac{x-y}{|x-y|} w_\d(x-y)\, dy = c_{n,s} F \int_{B(0,\d)}\frac{z \otimes z}{|z|^{n+s+1}} w_\d(z)\, dz .
 \]
Let $1 \leq i, j \leq n$.
If $i \neq j$, we have 
\[
 \int_{B(0,\d)}\frac{z_i z_j}{|z|^{n+s+1}} w_\d(z)\, dz = 0
\]
as can be seen by performing a reflection along the $x_i$ axis.
On the other hand, performing a rotation taking the $x_i$ axis onto the $x_1$ axis, we can see that 
\[
 \int_{B(0,\d)}\frac{z_i^2}{|z|^{n+s+1}} w_\d(z)\, dz = \int_{B(0,\d)}\frac{z_1^2}{|z|^{n+s+1}} w_\d(z)\, dz ,
\]
so
\[
 \int_{B(0,\d)}\frac{z_i^2}{|z|^{n+s+1}} w_\d(z)\, dz = \frac{1}{n} \sum_{j=1}^n \int_{B(0,\d)}\frac{z_j^2}{|z|^{n+s+1}} w_\d(z)\, dz = \frac{1}{n} \int_{B(0,\d)}\frac{1}{|z|^{n+s-1}} w_\d(z)\, dz = \frac{1}{c_{n,s}} ,
\]
thanks to the normalization \ref{item:wnormalization} of Section \ref{se:functional}.
All in all,
\[
 \int_{B(0,\d)}\frac{z \otimes z}{|z|^{n+s+1}} w_\d(z)\, dz = \frac{1}{c_{n,s}} I ,
\]
and, hence, $D_\d^s u (x) = F$. Now, the fact $u \in H^{s,p} (\O, \Rn)$ is clear from Definition \ref{de:Hspd}.
\end{proof}

\section{Linearization of the nonlocal equations}\label{se:linearization}
 
In \cite[Th.\ 8.2]{BeCuMo22} and \cite[Th.\ 6.2]{BeCuMo23} we showed the Euler--Lagrange equation corresponding to the minimization problem.
In our case, if the body force is $f : \O \to \Rn$ and the boundary condition is (for simplicity) $u |_{\O_{\d} \setminus \O_{-\d}} = \id$, the total energy of $u \in H^{s,p,\d}_{\id} (\O_{-\d}, \Rn)$ is
\[
 \int_{\O} W (x, D^s_{\d} u) \, dx - \int_{\O} f \cdot u \, dx .
\]
Thus, if we denote by $T_R (x,F) = D_F W(x,F)$ the Piola--Kirchhoff stress tensor, the Euler--Lagrange equation is
\begin{equation}\label{eq:EL}
 - \diver^s_{\d} T_R (x, D^s_{\d} u (x)) = f(x) , \qquad x \in \O_{-\d} .
\end{equation}
The (formal) linearization of this equation is totally analogous to that in the classical case
\[
 - \diver T_R (x, D u (x)) = f(x) , \qquad x \in \O ;
\]
see, e.g., \cite[Ch.\ X]{Gurtin81}, \cite[Ch.\ 4]{MaHu94} or \cite[Ch.\ 52]{GuFrAn10}.
As the operator $\diver^s_{\d}$ is linear, the main issue in the linearization of \eqref{eq:EL} lies on $T_R (x, D^s_{\d} u (x))$, but, then, for the linearization of $T_R (x, D^s_{\d} u (x))$ it is enough to substitute $D u$ of the classical theory by $D^s_{\d} u$.

To be more precise, and assuming the usual simplification of the classical theory, we linearize around the identity map and assume that $W$ does not depend on $x$.
The identity is supposed to be stress-free, so $T_R (I) = 0$.
We write $u = \id + v$, where $v$ is the displacement.
For the linearization process, we assume that $D^s_{\d} v$ is small.
We denote by $C$ the elasticity tensor, with components
\[
 c_{ijkl} = \frac{\partial^2 W}{\partial F_{ij} \partial F_{kl}} (I) , \qquad 1 \leq i, j, k, l \leq n .
\]
Thus, by the classical theory, the linearization of $T_R (D^s_{\d} u)$ is $C D^s_{\d} v$.
Therefore, the linearization of \eqref{eq:EL} is
\[
 - \diver^s_{\d} \left( C D^s_{\d} v \right) (x) = f(x) , \qquad x \in \O_{-\d} 
\]
with boundary conditions $v|_{\O_{\d} \setminus \O_{-\d}} = 0$.
As in the classical case,  $C D^s_{\d} v =  C D^s_{\d,\sym} v$.

In conclusion, the strong form of the equations of nonlocal linear elasticity are
\begin{equation}\label{eq:lineareq}
 \begin{cases}
 - \diver^s_{\d} \left( C D^s_{\d,\sym} v \right) = f & \text{ in } \O_{-\d} , \\
 v = 0 & \text{ in } \O_{\d} \setminus \O_{-\d} ,
 \end{cases}
\end{equation}
while its weak form is:
\begin{equation}\label{eq:weaklinearnonlocal}
 \text{Find } v \in H^{s,2,\d}_0 (\O_{-\d}, \Rn) \text{ such that } \int_{\O} C D^s_{\d,\sym} v : D^s_{\d,\sym} w = \int_{\O} f \cdot w , \quad \forall w \in H^{s,2,\d}_0 (\O_{-\d}, \Rn).
\end{equation}
This weak form is obtained in the usual fashion, multiplying by a test function, and then performing (nonlocal) integration by parts with the use of Proposition \ref{th:Nl parts}.

Of course, the process of linealization has been purely formal.
A rigorous derivation by $\G$-convergence would entail an adaptation of \cite{DaNePe02} and, in particular, the rigidity estimates of \cite{FrJaMu02}.

When $W$ is isotropic, the elasticity tensor takes the simpler expression
\[
 C e = 2 \mu e + \lambda (\tr e) I , \qquad e \in \Rnn_{\sym} ,
\]
where $\mu, \lambda \in \R$ are the Lam\'e moduli.
Positive definiteness of $C$ entails $\mu>0$ and $2 \mu + n \lambda >0$.
In this case,
\begin{equation}\label{eq:CDv}
 C D^s_{\d,\sym} v = 2 \mu D^s_{\d,\sym} v + \lambda (\tr D^s_{\d,\sym} v) I = 2 \mu D^s_{\d,\sym} v + \lambda (\diver^s_{\d} v) I
\end{equation}
and
\[
 C D^s_{\d,\sym} v : D^s_{\d,\sym} w = 2 \mu  D^s_{\d,\sym} v : D^s_{\d,\sym} w + \lambda \diver^s_{\d} v \diver^s_{\d} w ,
\]
where we have used Lemma \ref{le:trD}.
Therefore, the weak form of the equations of isotropic nonlocal linear elasticity is to find $v \in H^{s,2,\d}_0 (\O_{-\d}, \Rn)$ such that for all $w \in H^{s,2,\d}_0 (\O_{-\d}, \Rn)$,
\[
 \int_{\O} \left( 2 \mu  D^s_{\d,\sym} v : D^s_{\d,\sym} w + \lambda \diver^s_{\d} v \diver^s_{\d} w \right) = \int_{\O} f \cdot w .
\]

The strong form requires a longer argument, which assumes that $v \in C^{\infty}_c (\O_{-\d}, \Rn)$.
Starting from \eqref{eq:CDv}, we note that
\[
 \diver^s_{\d} \left( C D^s_{\d,\sym} v \right) = 2 \mu \diver^s_{\d} \left( D^s_{\d,\sym} v \right) + \lambda \diver^s_{\d} \left( (\diver^s_{\d} v) I \right),
\]
Now, thanks to \cite[Lemma 3.4 and Def.\ 3.1]{BeCuMo23}
\[
 \diver^s_{\d} \left( (\diver^s_{\d} v) I \right) = D^s_{\d} (\diver^s_{\d} v) .
\]
At this stage, we ought to define the nonlocal Laplacian $\Delta^s_{\d}$.
\begin{definition}\label{de:nonlocallaplacian}
 For $u \in C^{\infty}_c (\Rn)$ or $u \in C^{\infty}_c (\Rnn)$, we define 
\[
 \Delta^s_{\d} u = \diver^s_{\d} D^s_{\d} u .
\]
\end{definition}
It is not the scope of this work to study this operator, not even to write down a pointwise definition of it involving only one integral, in the spirit of Definition \ref{def: nonlocal gradient}.
We just define $\Delta^s_{\d}$ in this way for parallelism with the classical case and to continue with the analysis of the linear isotropic equation.

With Definition \ref{de:nonlocallaplacian} and Lemma \ref{le:divDT} we have 
\[
 \diver^s_{\d} \left( D^s_{\d,\sym} v \right) = \frac{\diver^s_{\d} \left( D^s_{\d} v \right) + \diver^s_{\d} \left( D^s_{\d} v \right)^T }{2} = \frac{\Delta^s_{\d} u +  D^s_{\d} (\diver^s_{\d} u)}{2} .
\]
All in all, we have that
\[
 \diver^s_{\d} \left( C D^s_{\d,\sym} v \right) = \mu \Delta^s_{\d} v + (\mu + \lambda) D^s_{\d} (\diver^s_{\d} v) ,
\]
and, hence, the strong form of the nonlocal linear equations in the isotropic case are
\begin{equation*}
 \begin{cases}
 - \mu \Delta^s_{\d} v - (\mu + \lambda) D^s_{\d} (\diver^s_{\d} v) = f & \text{ in } \O_{-\d} , \\
 v = 0 & \text{ in } \O_{\d} \setminus \O_{-\d} .
 \end{cases}
\end{equation*}

\section{Existence and uniqueness of solutions}\label{se:existence}

We present Korn's inequality in $H^{s,p,\d}$.

\begin{prop}\label{prop:Korn}
Let $1 < p < \infty$.
Then there exists $c>0$ such that for all $u \in H^{s,p,\d} (\O, \Rn)$,
\[
 \left\| D^s_{\d,\sym} u \right\|_{L^p (\O)} \geq c \left\| D^s_{\d} u \right\|_{L^p (\O)} .
\]
\end{prop}
\begin{proof}
By Theorem \ref{theo:DsdD}, $D^s_{\d} u = D v$ for some $v \in W^{1,p} (\O, \Rn)$, so $D^s_{\d,\sym} u = D_{\sym} v$.
By Korn's inequality (e.g., \cite[Sect.\ 3.1]{AcDu17}),
\[
 \left\| D^s_{\d,\sym} u \right\|_{L^p (\O)} = \left\| D_{\sym} v \right\|_{L^p (\O)} \geq c \left\| D v \right\|_{L^p (\O)} = c \left\| D^s_{\d} u \right\|_{L^p (\O)} .
\]
\end{proof}

Now we show the existence and uniqueness of weak solutions of equation \eqref{eq:lineareq}.
As in classical linear elasticity, one assumes that the tensor $C$, viewed as a linear operator from $\Rnn_{\sym}$ to $\Rnn_{\sym}$, is positive definite.

\begin{theo}
Assume $C : \Rnn_{\sym} \to \Rnn_{\sym}$ is positive definite.
Let $f \in L^2 (\O, \Rn)$.
Then there exists a unique solution to \eqref{eq:weaklinearnonlocal}.
Moreover, it is the unique minimizer of the functional
\[
 E(v) = \frac{1}{2}\int_{\O} C D^s_{\d,\sym} v : D^s_{\d,\sym} v - \int_{\O} f \cdot v , \qquad v \in H^{s,2,\d}_0 (\O_{-\d}, \Rn) .
\]
\end{theo}
\begin{proof}
It is a standard application of Lax--Milgram's lemma.
The bilinear form
\[
 (v, w) \mapsto \int_{\O} C D^s_{\d,\sym} v : D^s_{\d,\sym} w
\]
is clearly continuous in $H^{s,2,\d} (\O, \Rn)$, hence in $H^{s,2,\d}_0 (\O_{-\d}, \Rn)$.
Moreover, it is symmetric.
The linear form
\[
 w \mapsto \int_{\O} f \cdot w
\]
is also continuous in $H^{s,2,\d} (\O, \Rn)$, hence in $H^{s,2,\d}_0 (\O_{-\d}, \Rn)$.
It is, then, enough to check the coercivity of the bilinear form.

Since $C$ is positive definite in $\Rnn_{\sym}$, there exists $c_1>0$ such that
\begin{equation}\label{eq:Hsym}
 C e : e \geq c_1 |e|^2 , \qquad e \in \Rnn_{\sym}.
\end{equation}
So, for all $v \in H^{s,2,\d}_0 (\O_{-\d}, \Rn)$,
\[
 \int_{\O} C D^s_{\d,\sym} v : D^s_{\d,\sym} v \geq c_1 \int_{\O} \left| D^s_{\d,\sym} v \right|^2 \geq c_2 \int_{\O} \left| D^s_{\d} v \right|^2 \geq c_3 \left\| v \right\|^2_{H^{s,p,\d}_0 (\O_{-\d}, \Rn)},
\]
for other constants $c_2, c_3 >0$, thanks to Proposition \ref{prop:Korn} and Theorem \ref{th:PoincareSobolev delta}.
The proof is complete.
\end{proof}

In order to apply Lax--Milgram's lemma in theorem above, we do not need that $f \in L^2 (\O, \Rn)$ but only that $f$ lies in the dual of $H^{s,2,\d}_0 (\O_{-\d}, \Rn)$, which is a space that has not been studied yet, but, at least, thanks to Theorem \ref{th:PoincareSobolev delta}, we can assert that it contains $L^q (\O, \Rn)$ for
\begin{equation}\label{eq:q}
\begin{cases}
 q\in \left[\frac{2 n}{n + 2 s} , \infty \right] & \text{if } 2 s <n , \\
 q\in (1, \infty] & \text{if } 2 s =n , \\
 q\in[1,\infty] & \text{if } 2 s >n .
\end{cases}
\end{equation}

\section{Existence and uniqueness of the linear fractional equation}\label{se:fractional}

When it comes to adapting the process of linearization of Section \ref{se:linearization} to the case of the Bessel spaces $H^{s,p} (\Rn, \Rn)$ in terms of the Riesz derivative $D^s$ instead of the nonlocal derivative $D^s_{\d}$, one finds the obstacle that the identity does not belong to $H^{s,p} (\Rn, \Rn)$ given the lack of integrability in $\Rn$. 
%
However, if we could assume that the linearization of $T_R (D^s u)$ is $C D^s v$,
then the process would be completed verbatim as in Sections \ref{se:linearization} and \ref{se:existence}.
	
An account of the main properties of the space $H^{s,p}$, the Riesz derivative $D^s$ and related fractional operators can be found in \cite{ShSp15,ShSp17,CoSt19,BeCuMo20} and \cite[Sect.\ 15.2]{Ponce16}.
In fact, formally $D^s$ is an extreme case of $D^s_{\d}$ when $\delta =\infty$ and $w_\d = 1$.

We recall the Euler--Lagrange equations found in \cite[Th.\ 6.2]{BeCuMo20}:
\[
\begin{cases}
 - \diver^s T_R (x, D^s u (x)) = f(x) & \text{ in } \O , \\
 u = \id & \text{ in } \O^c .
\end{cases}
\]
The expected linearization then becomes
\begin{equation}\label{eq:lineareqRiesz}
 \begin{cases}
 - \diver^s C D^s_{\sym} v = f & \text{ in } \O , \\
 v = 0 & \text{ in } \O^c ,
 \end{cases}
\end{equation}
while its weak form is:
\begin{equation}\label{eq:weaklinearfractional}
 \text{Find } \ v \in H^{s,2}_0 (\O, \Rn) \ \text{ such that } \ \int_{\Rn} C D^s_{\sym} v : D^s_{\sym} w = \int_{\O} f \cdot w , \qquad \forall w \in H^{s,2}_0 (\O, \Rn).
\end{equation}

We recall that $H^{s,2}_0 (\O, \Rn)$ is the closure in $H^{s,2} (\Rn, \Rn)$ of $C^{\infty}_c (\O, \Rn)$.
To obtain this weak formulation one uses the fractional integration by parts formula \cite[Th.\ 3.6]{BeCuMo20}.

We first need the following well-known version of Korn's inequality in $W^{1,p} (\Rn, \Rn)$.
We have not found an identical statement in the literature, so we prove it by adapting another version of the inequality.

\begin{lem}\label{le:KornRn}
Let $1 < p < \infty$.
Then there exists $c_p > 0$ such that for all $u \in W^{1,p} (\Rn, \Rn)$,
\[
 \left\| D_{\sym} u \right\|_{L^p (\Rn)} \geq c_p \left\| D u \right\|_{L^p (\Rn)} .
\]
\end{lem}
\begin{proof}
It is enough to prove the inequality for $u \in C^{\infty}_c (\Rn, \Rn)$.
Let $R>0$ be such that $\supp u \subset B(0,R)$.
Then the function $v \in C^{\infty}_c (\Rn, \Rn)$ defined by $v (x) = u (Rx)$ satisfies $\supp v \subset B(0,1)$.
By Korn's inequality (see, e.g., \cite[Th.\ 7.71]{DeDe12}) as well as Poincar\'e's, there exists $c_p >0$ such that
\[
 \left\| D_{\sym} v \right\|_{L^p (\Rn)} \geq c_p \left\| D v \right\|_{L^p (\Rn)} .
\]
Now, it is immediate to see that
\[
 \left\| D v \right\|_{L^p (\Rn)} = R^{1-\frac{n}{p}} \left\| D u \right\|_{L^p (\Rn)} \quad \text{and} \quad \left\| D_{\sym} v \right\|_{L^p (\Rn)} = R^{1-\frac{n}{p}} \left\| D_{\sym} u \right\|_{L^p (\Rn)} ,
\]
and, hence, the result follows.
\end{proof}

The relevant Korn's inequality in $H^{s,p}$ is as follows.
For this we recall that, for $0<\alpha<n$, the Riesz potential $I_{\alpha} : \Rn \setminus \{0\} \to \R$ is defined as
\begin{equation}\label{eq:Rieszalpha}
 I_{\alpha} (x) = \frac{1}{\gamma(\alpha)} \frac{1}{|x|^{n-\alpha}},
\end{equation}
with $\gamma$ defined by equality \eqref{eq:gamma}.

\begin{prop}\label{prop:KornRiesz}
Let $1 < p < \infty$.
Then there exists $c_p > 0$ such that for  all $u \in H^{s,p} (\Rn, \Rn)$,
\[
 \left\| D^s_{\sym} u \right\|_{L^p (\Rn)} \geq c_p \left\| D^s u \right\|_{L^p (\Rn)} .
\]
\end{prop}
\begin{proof}
It is enough to prove the inequality for $u \in C^{\infty}_c (\Rn, \Rn)$.
For such $u$ we define $v = I_{1-s} * u$, where $I_{1-s}$ is Riesz' potential.
Then, it is known that $v \in W^{1,p} (\Rn, \Rn)$ and $D v = D^s u$.
A precise statement of this can be found in \cite[Eq.\ (2.4)]{KrSc22}, although the result is earlier (see, e.g., \cite[Prop.\ 5.2]{Silhavy20}, \cite[Prop.\ 2.2]{CoSt19} and \cite[Th.\ 1.2]{ShSp15}).
Then, by Lemma \ref{le:KornRn},
\[
 \left\| D^s_{\sym} u \right\|_{L^p (\Rn)} = \left\| D_{\sym} v \right\|_{L^p (\Rn)} \geq c_p \left\| D v \right\|_{L^p (\Rn)} = c_p \left\| D^s u \right\|_{L^p (\Rn)} .
\]
\end{proof}

The existence and uniqueness of weak solutions of equation \eqref{eq:lineareqRiesz} is as follows.

\begin{theo}\label{th:wellposedfractional}
Assume $C : \Rnn_{\sym} \to \Rnn_{\sym}$ is positive definite.
Let $f \in L^2 (\O, \Rn)$.
Then there exists a unique weak solution to \eqref{eq:weaklinearfractional}.
Moreover, it is the unique minimizer of the functional
\[
 E(v) = \frac{1}{2}\int_{\Rn} C D^s_{\sym} v : D^s_{\sym} v - \int_{\O} f \cdot v , \qquad v \in H^{s,2}_0 (\O, \Rn) .
\]
\end{theo}
\begin{proof}
It is enough to check the coercivity of the bilinear form
\[
 (v, w) \mapsto \int_{\Rn} C D^s_{\sym} v : D^s_{\sym} w
\]
in $H^{s,2}_0 (\O, \Rn)$.

Using \eqref{eq:Hsym}, Proposition \ref{prop:KornRiesz} and the Poincar\'e inequality of \cite[Th.\ 1.8]{ShSp15} (see also \cite[Th.\ 2.3]{BeCuMo20}) we find that for all $v \in H^{s,2}_0 (\O, \Rn)$,
\[
 \int_{\Rn} C D^s_{\sym} v : D^s_{\sym} v \geq c_1 \int_{\Rn} \left| D^s_{\sym} v \right|^2 \geq c_2 \int_{\Rn} \left| D^s v \right|^2 \geq c_3 \left\| v \right\|^2_{H^{s,2}_0 (\O, \Rn)},
\]
for some constants $c_1, c_2, c_3 >0$.
The proof is complete.
\end{proof}

As in Section \ref{se:existence}, it is enough that $f$ lies in the dual of $H^{s,2}_0 (\O, \Rn)$, which, thanks to \cite[Th.\ 1.8]{ShSp15} (see also \cite[Th.\ 2.3]{BeCuMo20}), contains $L^q (\O, \Rn)$ for $q$ as in \eqref{eq:q}.

Now we deal with the isotropic case, so as to write down the weak and strong forms of the equations.

The weak form of the equations of isotropic fractional linear elasticity is to find $v \in H^{s,2}_0 (\O, \Rn)$ such that for all $w \in H^{s,2}_0 (\O, \Rn)$,
\[
 \int_{\O} \left( 2 \mu  D^s_{\sym} v : D^s_{\sym} w + \lambda \diver^s v \diver^s w \right) = \int_{\O} f \cdot w ,
\]
where we have used $\tr D^s = \diver^s$ (see Proposition \ref{pr:fractionalidentities} below).

Regarding the strong form, we assume that $v \in C^{\infty}_c (\O, \Rn)$.
As in Section \ref{se:linearization}, if we prove the identities
\begin{equation}\label{eq:fractionalidentities}
\begin{split}
 & \tr D^s v = \diver^s v , \qquad \diver^s \left( (\diver^s v) I \right) = D^s (\diver^s v) , \\
 & \diver^s D^s v = \Delta^s v , \qquad \diver^s (D^s v)^T = D^s (\diver^s v) ,
\end{split}
\end{equation}
then the strong form of the fractional linear equations in the isotropic case are
\begin{equation*}
 \begin{cases}
 - \mu \Delta^s v - (\mu + \lambda) D^s (\diver^s v) = f & \text{ in } \O , \\
 v = 0 & \text{ in } \O^c .
 \end{cases}
\end{equation*}

It turns out that identities \eqref{eq:fractionalidentities} have been proved or are immediate consequences of known results, as we now show.

\begin{prop}\label{pr:fractionalidentities}
Equalities \eqref{eq:fractionalidentities} hold for any $v \in C^{\infty}_c (\O, \Rn)$.
\end{prop}
\begin{proof}
Equality
\[
 \tr D^s v = \diver^s v
\]
can be proved as in \cite[Lemma 3.4]{BeCuMo23}.
Equality
\[
 \diver^s \left( (\diver^s v) I \right) = D^s (\diver^s v)
\]
can be proved easily by using \cite[Lemma 3.7]{BeCuMo20}.
Equality
\[
 \diver^s D^s v = \Delta^s v
\]
has been proved in \cite[Th.\ 5.3]{Silhavy20} and \cite[Th.\ 3.2]{DeGuOlKa21}, although the result is earlier.
Finally, equality
\[
 \diver^s (D^s v)^T = D^s (\diver^s v)
\]
can be proved as in Lemma \ref{le:divDT}.
\end{proof}

We mention that Proposition \ref{prop:KornRiesz} and Theorem \ref{th:wellposedfractional} were proved, with different perspective and techniques, by \cite{Silhavy22}, where the equations of fractional linear elasticity were considered without a detailed derivation.
We have included a proof here to stress that the techniques of the nonlocal case can be adapted to the fractional case, and vice versa.
In fact, a thorough study of the isotropic fractional linear elasticity equations has been done in \cite{Scott23}, where it is included a derivation of the equation from a different point of view.

 \section{Connection with Eringen's model}\label{se:Eringen}

It is interesting to compare the linearisation of this model with other nonlocal linear models that have appeared in the literature.
Here is where the relationship with nonlocal Eringen's model \cite{EringenBook} shows up. This model, which has been used with popularity among the engineering community, turns out to be in general ill-posed for smooth kernels, whereas for non-smooth kernels like the Riesz potential it does admit a solution, as shown in \cite{EvBe}.
Initially we introduce the model considered in \cite{EvBe}, which follows \cite{Polizzotto}, without detailing the functional framework:
 \begin{equation} \label{eq: Eringen model}
  \begin{cases}
  -\diver \sigma =f,  &\text{ in $\Omega$} \\
  \sigma (x) =\int_{\Omega}A(x,x')CD_{\sym} v(x') \, dx', &\text{ in $\Omega$} \\
  v=\bar{v}, &\text{ on $\p \O$} .
  \end{cases}  
  \end{equation}
Here $C$ is the (positive definite) fourth-order elasticity tensor, as in Section \ref{se:linearization}, and $\sigma$ the nonlocal stress tensor, which is obtained through averaging  the local stress tensor $C D_{\sym}v$.
The term $A(x,x')$ is the nonlocal interaction kernel between particles accounting for stress interaction.
We assume that $A$ belongs to $L^1(\O\times\O)$ and is a strictly positive kernel, i.e., it satisfies the so-called Mercer's condition,
 \[\int_\O\int_\O A(x,x') \psi(x)\psi(x')\,dx\,dx' >0,\]
 for any $\psi\in C_c^\infty(\O)\setminus \{0\}$.
 
As for the corresponding functional framework, we notice that the symmetric bilinear form 
\[
 (v,w)_A = \int_\Omega\int_\Omega A(x,x')Dv(x):Dw(x')\,dx\,dx'
\]
defines an inner product over $C_c^\infty(\Omega,\Rn)$, and, consequently, induces a norm $\left\| \cdot \right\|_A$. We define the space $V_A$ as the completion of $C_c^\infty(\Omega,\Rn)$ with respect to the norm $\left\| \cdot \right\|_A$. 

Obtaining now the weak formulation of model \eqref{eq: Eringen model} is a standard task.
Given a function $v_0:\Omega\to\Rn$ satisfying the boundary condition, we multiply the equilibrium equation by a test function $w \in C_c^\infty(\Omega,\Rn)$ and integrate by parts to arrive at
  \begin{equation}\label{eq:weakEringen1}
  \begin{split}
 & \text{Find } \ v\in v_0+V_A \ \text{ such that } \ a(v,w)= \int_{\Omega} f \cdot w , \qquad \forall w \in C_c^\infty(\Omega,\Rn), \\
 & \text{where } \ a(v,w) := \int_{\Omega} \int_{\Omega} A(x,x')C D_{\sym}v(x) : D_{\sym}w(x') \, dx' \, dx .
 \end{split}
  \end{equation}
In fact, by density, the test function $w$ can be taken in $V_A$.
  
As is typical in Eringen's model, we will additionally assume that the kernel $A: \O \times \O \to [0, \infty)$ has the form $A(x,x')=\tilde A(|x-x'|)$ for some function $\tilde A: (0, \infty)\to [0, \infty)$, so that $\tilde A (d)$ describes the nonlocal interaction of particles at a distance $d> 0$. 
For simplicity, we consider homogeneous Dirichlet boundary conditions, i.e., $\bar v=0$ in \eqref{eq: Eringen model}, hence $v_0=0$ in the previous weak formulation.
With a small abuse of notation, we assume that functions defined in $\O$ (such as functions in $C_c^\infty(\Omega,\Rn)$ or in $V_A$) are extended by zero outside $\Omega$ to the whole space.

With these premises, the weak formulation \eqref{eq:weakEringen1} becomes
\begin{equation}\label{eq:weakEringen2}
  \begin{split}
 & \text{Find } \ v\in V_A \ \text{ such that } \ a(v,w)= \int_{\Omega} f \cdot w , \qquad \forall w \in V_A, \\
 & \text{where } \ a(v,w) := \int_{\Rn} \int_{\Rn} \tilde A(|x-x'|) C D_{\sym}v(x) : D_{\sym}w(x') \, dx' \, dx.
 \end{split}
\end{equation}

Although this model does not seem to coincide with the linear models presented in Sections \ref{se:fractional} and \ref{se:linearization}, two particular choices of the interaction function $\tilde A$ give rise, respectively, to equivalent problems:
\[
 \tilde A (|x|) = I_{2(1-s)} (x), \quad \text{and} \quad \tilde A (|x|) = Q^s_{\d}\ast Q^s_{\d} (x),
\]
where $I_{2(1-s)}$ is Riesz' potential \eqref{eq:Rieszalpha} and $Q^s_{\d}$ is the potential of Theorem \ref{theo:DsdD}.
 

\subsection{Kernel $\tilde A$ as Riesz potential}\label{subse:KernelRiesz}
 
In \cite{EvBe}, it is shown that Eringen's model is well posed when the nonlocal interaction function $\tilde A$ is taken as the Riesz potential.
Here we will see the equivalence between this model and the fractional Navier equation \eqref{eq:lineareqRiesz}.

In order to see that, we explain the detailed framework.
Let $0<\alpha<n$ and recall the Riesz potential \eqref{eq:Rieszalpha}.
In formulation \eqref{eq:weakEringen2}, rename the kernel $A$ as $A_{\a}$, the interaction kernel $\tilde{A}$ as $\tilde{A}_{\a}$ and the space $V_A$ as $V_{A_{\a}}$, with
\[
 \tilde A_{\a} (|x|) = I_{\a} (x) ,
\]
so that model \eqref{eq:weakEringen2} becomes
\begin{equation}\label{eq:weakEringenfractional}
  \begin{split}
 & \text{Find } \ v\in V_{A_{\a}} \ \text{ such that } \ a_{\a} (v,w)= \int_{\Omega} f \cdot w , \qquad \forall w \in V_{A_{\a}}, \\
 & \text{where } \ a_{\a} (v,w) := \int_{\Rn} \int_{\Rn} I_{\a} \left(x-x' \right) C D_{\sym}v(x) : D_{\sym} w(x') \, dx' \, dx.
 \end{split}
\end{equation}
Comparing \eqref{eq:weaklinearfractional} and \eqref{eq:weakEringenfractional}, we can see that in order to prove that both models are equivalent it is enough to show that the bilinear forms and the spaces are equal.
Prior to do that, we need some facts about fractional Sobolev spaces.

The well-known fractional Sobolev space $H^s (\Rn, \Rn)$ \cite{Adams,DPV} is a Hilbert space endowed with the inner product 
 \[
 (u,v)_{H^s(\Rn,\Rn)} = (u,v)_{L^2 (\Rn,\Rn)} + \int_{\Rn} \int_{\Rn} \frac{(u(x)-u(x'))\cdot (v(x)-v(x'))}{|x-x'|^{n+2s}} \, dx' \, dx ,
\]
where $\left( \cdot, \cdot \right)_{L^2 (\Rn,\Rn)}$ is the inner product in $L^2 (\Rn,\Rn)$.
This space coincides (with equivalence of norms) with the Bessel space $H^{s,2}(\Rn,\Rn)$; see \cite[Ch.\ 7, pp.\ 221]{Adams}.
Therefore, the well known subspace $\tilde{H}^s(\Omega,\Rn)$, defined as the closure of $ C_c^\infty(\Omega,\Rn)$ in $H^s(\Rn,\Rn)$, coincides with $H^{s,2}_0(\Omega,\Rn)$ (see Section \ref{se:fractional}). For connections with the results in \cite{EvBe}, it is worth mentioning that $H_0^s(\Omega,\Rn)$, defined as the closure of $C_c^{\infty}(\Omega,\Rn)$ in $H^s(\Omega, \Rn)$, is isomorphic to $\tilde{H}^s(\Omega,\Rn)$ whenever $s\ne \frac{1}{2}$; see \cite[Th.\ 3.33]{mclean2000strongly}. 


The main result of this section establishes the equivalence between models \eqref{eq:weaklinearfractional} and \eqref{eq:weakEringenfractional} when $\a = 2 (1-s)$.

\begin{theo}\label{riesz-eringen}
Let $s\in\left(\max\{0,\frac{2-n}{2}\},1\right)$.
Then
\[
 V_{A_{2(1-s)}} = H^{s,2}_0(\Omega,\Rn)
\]
with equivalence of norms, and 
\begin{equation}\label{eq:bilinearriesz}
 a_{2(1-s)}(v,w)=\int_{\Rn} C D^s_{\sym} v(x) : D^s_{\sym} w (x) \, dx , \qquad v,w \in H_0^{s,2}(\Omega, \Rn) .
\end{equation}
 \end{theo}
\begin{proof}
Let $\a = 2 (1-s)$.
For $v \in C_c^\infty(\Omega,\Rn)$, its $s$-fractional gradient can be written as $D^s u = I_{1-s}*Du$ (see \cite[Th.\ 1.2]{ShSp15}).
Moreover,
\[
 \int_{\Rn} A_{\a} (x, x') D v (x) \, dx = \int_{\Rn} I_{\a} (x'-x) D v (x) \, dx = I_{\a} * Dv (x') , \qquad x' \in \Rn .
\]
With these equalities, making use of the fact $I_{\a} = I_{1-s} * I_{1-s}$ (the semigroup property for the Riesz potential) and
\[
 \int_{\Rn} \left( I_{\a} * g \right) h = \int_{\Rn} \left( I_{1-s} * I_{1-s} * g \right) h = \int_{\Rn} \left( I_{1-s} * g \right) \left( I_{1-s} * h \right), \qquad f, g \in C^{\infty}_c (\O)
\]
(see, e.g., \cite[Prop.\ 4.16]{Brezis}), we obtain
\begin{align*} 
 \left\| v \right\|_{A_{\a}} & = \int_{\Rn} \int_{\Rn} A_{\a} (x, x') D v (x) : D v (x') \, dx \, dx' = \int_{\Rn} I_{\a} * D v (x') : D v(x') \, dx' \\
 & = \int_{\Rn} I_{1-s} * D v (x') :  I_{1-s} * D v(x') \, dx' = \left\| D^s v \right\|_{L^2} .
\end{align*}
To finish, just notice that thanks to Poincar\'e's inequality in Bessel spaces (see \cite[Th.\ 2.2]{BeCuMo20} and the references therein), the seminorm $\|D^su\|_{L^2}$ is equivalent to the $H^{s,2}$ norm on the subspace $H^{s,2}_0(\Omega,\Rn)$.
This proves the equality $V_{A_{2(1-s)}} = H^{s,2}_0(\Omega,\Rn)$.

Now we show \eqref{eq:bilinearriesz}.
By density, it is enough to prove it for functions $v,w\in C_c^\infty(\O,\Rn)$.
Following the lines of the previous argument, we have
\begin{align*}
 a_{\a} (v,w) & = \int_{\Rn}\int_{\Rn}I_{\a} (x-x') C D_{\sym}v(x):D_{\sym}w(x')\,dx\,dx' \\
 & = \int_{\Rn} (I_{\a}\ast CD_{\sym}v)(x') :D_{\sym}w(x')\,dx' \\
 & = \int_{\Rn} (I_{1-s}\ast CD_{\sym}v)(x') : ( I_{1-s} \ast D_{\sym} )w(x')\,dx = \left( CD^s_{\sym}v,D^s_{\sym}w \right)_{L^2}.
\end{align*}
\end{proof}

\begin{rem} Previous proof may alternatively be done via Fourier transform. Namely, making use of Plancherel's equality and the fact that $\mathcal{F}(I_{\alpha}) (\xi) = |2\pi \xi|^{-\alpha}$, for $v,w\in C_c^\infty(\O,\Rn)$, we have
\begin{align*}
 a_{\a} (v,w) & = \int_{\Rn}\int_{\Rn}I_{\a} (x-x') C D_{\sym}v(x):D_{\sym}w(x')\,dx\,dx' \\
 & = \int_{\Rn} (I_{\a}\ast CD_{\sym}v)(x') :D_{\sym}w(x')\,dx'= \left( \mathcal{F}\{ I_{\a}\ast CD_{\sym}v\},\mathcal{F}\{D_{\sym}w\} \right)_{L^2} \\
 & = \left( |2\pi\xi|^{-2(1-s)}\mathcal{F}\{CD_{\sym}v\},\mathcal{F}\{D_{\sym}w\} \right)_{L^2} \\
 & = \left( |2\pi\xi|^{-(1-s)}\mathcal{F}\{CD_{\sym}v\},|2\pi\xi|^{-(1-s)}\mathcal{F}\{D_{\sym}w\} \right)_{L^2} \\
 & = \left( \mathcal{F}\{I_{1-s}\ast CD_{\sym}v\},\mathcal{F}\{I_{1-s}\ast D_{\sym}w\} \right)_{L^2} = \left( CD^s_{\sym}v,D^s_{\sym}w \right)_{L^2}.
\end{align*}
\end{rem}

We finish this section with a comparison with the results of \cite{EvBe}.
As mentioned before, $H_0^{s,2}(\Omega, \Rn)$ coincides with $H_0^s(\Omega,\Rn)$ for $s\neq \frac{1}{2}$, so Theorem \ref{riesz-eringen} recovers the result of \cite[Prop.\ 11]{EvBe} stating the equivalence between $V_{A_{2(1-s)}}$ and $H_0^s(\Omega,\Rn)$ for $s\neq \frac{1}{2}$.
Moreover, when $C$ is a positive definite tensor, Theorem \ref{th:wellposedfractional} proves the existence and uniqueness of solutions of \eqref{eq:weaklinearfractional} for any $s\in(0,1)$.
In \cite{EvBe}, it was shown using different techniques and another nonlocal Korn's inequality, existence and uniqueness of \eqref{eq:weakEringenfractional} for any $s\in(0,1)\setminus\{\frac{1}{2}\}$.

\subsection{Kernel $\tilde A$ as potential $Q^s_{\d}\ast Q^s_{\d}$}

In this section we consider the nonlocal linear elastic model \eqref{eq:weaklinearnonlocal} and will see that it is equivalent to Eringen's model when we consider as nonlocal interaction kernel the one determined by the function $Q^s_{\d}\ast Q^s_{\d}$, with $Q_{\d}^s$ the potential function from Theorem \ref{theo:DsdD}. Notice that, in opposition to Riesz potentials, the family of kernels $\{ Q_{\d}^s \}_{0 < s < 1}$ introduced in Theorem \ref{theo:DsdD} does not satisfy the semigroup property, so that convolution of different kernels of this kind is not in principle a kernel in the same class.

In the general Eringen model \eqref{eq:weakEringen2} we take the kernel $A$ as
\[
 A_{s,\d}(x,x') = (Q_{\d}^s\ast Q_{\d}^s) (x-x') .
\]
Since $Q_{\d}^s$ is radial and the convolution of radial functions is radial, $A_{s,\d}(x,x')$ depends on the modulus $|x-x'|$, and is strictly positive in the interior of its support by Theorem \ref{theo:DsdD}.
Moreover, $A_{s, \d} \in L^1 (\O \times \O)$ thanks to Young's inequality.
Therefore, the interaction function $\tilde{A}$ is here
\[
 \tilde{A}_{s, \d} (|x|) = Q_{\d}^s\ast Q_{\d}^s (x) .
\]
The bilinear form $A$ and the inner product $(\cdot, \cdot)_{V_A}$ are, respectively,
\begin{align*}
 & a_{s,\d}(v,w)=\int_{\O_{-\d}}\int_{\O_{-\d}} A_{s,\d}(x,x') CD_{\sym}v(x):D_{\sym}w(x')\,dx\,dx', \\
 & (v,w)_{V_{A_{s,\d}}}=\int_{\O_{-\d}}\int_{\O_{-\d}} A_{s,\d}(x,x') Dv(x):Dw(x')\,dx\,dx'.
\end{align*}
Similarly to Section \ref{subse:KernelRiesz}, we define the space $V_{A_{s,\d}}$ as the closure of $C_c^\infty(\O_{-\d},\Rn)$ with respect to the inner product $(\cdot,\cdot)_{V_{A_{s,\d}}}$.
With a small abuse of notation, we assume that functions defined on $\O_{-\d}$ are extended by zero first to $\O_{\d}\setminus\O_{-\d}$ and then to $\Rn$.
This is consistent with the fact that $\supp D^s_\d u \subset \O$ for $u \in C_c^\infty(\O_{-\d},\Rn)$, so that $\supp D^s_\d u$ can be extended by zero to $\Rn \setminus \O$.

The strong formulation of Eringen's problem is
\begin{equation} \label{eq: Eringen model Q}
  \begin{cases}
  -\diver \sigma =f,  &\text{ in $\Omega_{-\d}$} \\
  \sigma (x) =\int_{\Omega_{-\d}} A_{s, \d} (x,x')CD_{\sym} v(x')dx', &\text{ in $\Omega_{-\d}$} \\
  v=0, &\text{ on $\partial\O_{-\d}$,}
  \end{cases}  
  \end{equation}
and the weak formulation is:
\begin{equation}\label{eq:weakeringenQ}
 \text{Find } \ v\in V_{A_{s,\d}} \ \text{ such that } \ a_{s,\d} (v,w)= \int_{\Omega} f \cdot w , \qquad \forall w \in V_{A_{s,\d}}, \\
\end{equation}
for a given $f\in L^2(\O,\Rn)$.

The equivalence of problems \eqref{eq:weakeringenQ} and \eqref{eq:weaklinearnonlocal} follows exactly the steps of Section \ref{subse:KernelRiesz}.
The following result has a proof analogous to that of Theorem \ref{riesz-eringen}, but in this case invoking Poincar\'e's inequality given by Theorem \ref{th:PoincareSobolev delta}. 

\begin{theo} \label{th:nonlocal Hilbert spaces}
The  identity of spaces 
\[V_{A_{s,\d}}=H_0^{s,2,\delta}(\O_{-\d},\Rn)\]
holds with equivalence of norms.
Moreover,
\begin{equation}\label{eq:asd}
 a_{s,\d}(v,w)=\int_{\O} CD^s_{\d,\sym} v(x): D^s_{\d,\sym}w(x)\,dx, \qquad v,w\in H_0^{s,2,\d}(\O_{-\d}, \Rn) .
\end{equation}
\end{theo}
%
\begin{proof}
For a function $v\in C_c^\infty(\O_{-\d},\Rn)$, using that $D^s_{\d}v=Q_{\d}^s\ast Dv$ (Theorem \ref{theo:DsdD}) and \cite[Prop.\ 4.16]{Brezis}, we have
\begin{align*} 
 (v,v)_{A_{s,\d}} & =\int_{\Rn}\int_{\Rn} A_{s,\d}(x,x') Dv(x):Dv(x')\,dx\,dx' = \int_{\Rn} \left[(Q_{\d}^s\ast Q_{\d}^s)\ast Dv\right](x') :Dv(x')\,dx' \\
 & = \int_{\Rn} \left( Q_{\d}^s \ast Dv \right)(x') : \left( Q_{\d}^s\ast Dv \right) (x')\,dx' = \left\| D^s_{\d} v \right\|_{L^2(\Rn, \Rnn)}^2 = \left\| D^s_{\d} v \right\|_{L^2(\O, \Rnn)}^2.
\end{align*}
Now, taking into account that, by Theorem \ref{th:PoincareSobolev delta}, the seminorm $\| D^s_{\d}v\|_{L^2(\O, \Rnn)}$ is equivalent to the norm of $H^{s,2,\d}(\O,\Rn)$ in the subspace $H^{s,2,\d}_0(\O_{-\d},\Rn)$, we conclude the result.

The proof of \eqref{eq:asd} follows the same steps to that of Theorem \ref{riesz-eringen}, but using also the arguments used above. 
\end{proof}

We notice that the functional spaces considered in Theorems \ref{riesz-eringen} and \ref{th:nonlocal Hilbert spaces} are actually the same, given the equality $H^{s,p,\d}_0(\O_{-\d})=H^{s,p}_0(\O_{-\d})$ shown in \cite[Lemma 2.16]{CuKrSc22}. Therefore, both approaches are defined on the same functional space, but working with different operators leading to different models.


\section*{Acknowledgements} 

This work has been supported by the Agencia Estatal de Investigaci\'on of the Spanish Ministry of Research and Innovation, through project PID2020-116207GB-I00 (J.C.B. and J.C.), and PID2021-124195NB-C32 and the Severo Ochoa Programme for Centres of Excellence in R\&D CEX2019-000904-S (C.M.-C.), by Junta de Comunidades de Castilla-La Mancha through project SBPLY/19/180501/000110 and European Regional Development Fund 2018/11744 (J.C.B. and J.C.), by the Madrid Government (Comunidad de Madrid, Spain) under the multiannual Agreement with UAM in the line for the Excellence of the University Research Staff in the context of the V PRICIT (Regional Programme of Research and Technological Innovation) (C.M.-C.), by the ERC Advanced Grant 834728 (C.M.-C.), and by Fundaci\'on Ram\'on Areces (J.C.).


\begin{thebibliography}{10}

\bibitem{AcDu17}
{\sc G.~Acosta and R.~G. Dur\'{a}n}, {\em Divergence Operator and Related
  Inequalities}, SpringerBriefs in Mathematics, Springer, New York, 2017.

\bibitem{Adams}
{\sc R.~A. Adams}, {\em {S}obolev spaces}, vol.~65 of Pure and Applied
  Mathematics, Academic Press, New York-London, 1975.

\bibitem{Ball77}
{\sc J.~M. Ball}, {\em Convexity conditions and existence theorems in nonlinear
  elasticity}, Arch. Rational Mech. Anal., 63 (1977), pp.~337--403.

\bibitem{BeCuMo20}
{\sc J.~C. Bellido, J.~Cueto, and C.~Mora-Corral}, {\em Fractional {P}iola
  identity and polyconvexity in fractional spaces}, Ann. Inst. H. Poincar\'{e}
  C Anal. Non Lin\'{e}aire, 37 (2020), pp.~955--981.

\bibitem{BeCuMo22}
\leavevmode\vrule height 2pt depth -1.6pt width 23pt, {\em Nonlocal gradients
  in bounded domains motivated by continuum mechanics: Fundamental theorem of
  calculus and embeddings}.
\newblock ArXiv preprint 2201.08793, 2022.

\bibitem{BeCuMo23}
\leavevmode\vrule height 2pt depth -1.6pt width 23pt, {\em Minimizers of
  nonlocal polyconvex energies in nonlocal hyperelasticity}.
\newblock ArXiv preprint 2211.02640, 2023.

\bibitem{Brezis}
{\sc H.~Brezis}, {\em Functional analysis, {S}obolev spaces and partial
  differential equations}, Universitext, Springer, New York, 2011.

\bibitem{CoSt19}
{\sc G.~E. Comi and G.~Stefani}, {\em A distributional approach to fractional
  {S}obolev spaces and fractional variation: existence of blow-up}, J. Funct.
  Anal., 277 (2019), pp.~3373--3435.

\bibitem{CuKrSc22}
{\sc J.~Cueto, C.~Kreisbeck, and H.~Sch\"onberger}, {\em Variational analysis
  of integral functionals involving nonlocal gradients on bounded domains}.
\newblock ArXiv preprint 2302.05569, 2023.

\bibitem{DaNePe02}
{\sc G.~Dal~Maso, M.~Negri, and D.~Percivale}, {\em Linearized elasticity as
  {$\Gamma$}-limit of finite elasticity}, Set-Valued Anal., 10 (2002),
  pp.~165--183.

\bibitem{DeGuMeSc21}
{\sc M.~D'Elia, M.~Gulian, T.~Mengesha, and J.~M. Scott}, {\em Connections
  between nonlocal operators: from vector calculus identities to a fractional
  {H}elmholtz decomposition}, Fract. Calc. Appl. Anal., 25 (2022),
  pp.~2488--2531.

\bibitem{DeGuOlKa21}
{\sc M.~D'Elia, M.~Gulian, H.~Olson, and G.~E. Karniadakis}, {\em Towards a
  unified theory of fractional and nonlocal vector calculus}, Fract. Calc.
  Appl. Anal., 24 (2021), pp.~1301--1355.

\bibitem{DeDe12}
{\sc F.~Demengel and G.~Demengel}, {\em Functional spaces for the theory of
  elliptic partial differential equations}, Universitext, Springer, London; EDP
  Sciences, Les Ulis, 2012.

\bibitem{DPV}
{\sc E.~Di~Nezza, G.~Palatucci, and E.~Valdinoci}, {\em Hitchhiker's guide to
  the fractional {S}obolev spaces}, Bull. Sci. Math., 136 (2012), pp.~521--573.

\bibitem{EringenBook}
{\sc A.~C. Eringen}, {\em Nonlocal continuum field theories}, Springer-Verlag,
  New York, 2002.

\bibitem{EvBe}
{\sc A.~Evgrafov and J.~C. Bellido}, {\em From non-local {E}ringen's model to
  fractional elasticity}, Math. Mech. Solids, 24 (2019), pp.~1935--1953.

\bibitem{FrJaMu02}
{\sc G.~Friesecke, R.~D. James, and S.~M\"uller}, {\em A theorem on geometric
  rigidity and the derivation of nonlinear plate theory from three-dimensional
  elasticity}, Comm. Pure Appl. Math., 55 (2002), pp.~1461--1506.

\bibitem{Gurtin81}
{\sc M.~E. Gurtin}, {\em An introduction to continuum mechanics}, vol.~158 of
  Mathematics in Science and Engineering, Academic Press, New York-London,
  1981.

\bibitem{GuFrAn10}
{\sc M.~E. Gurtin, E.~Fried, and L.~Anand}, {\em The mechanics and
  thermodynamics of continua}, Cambridge University Press, Cambridge, 2010.

\bibitem{KrSc22}
{\sc C.~Kreisbeck and H.~Sch\"onberger}, {\em Quasiconvexity in the fractional
  calculus of variations: characterization of lower semicontinuity and
  relaxation}, Nonlinear Anal., 215 (2022), pp.~112625, 26.

\bibitem{MaHu94}
{\sc J.~E. Marsden and T.~J.~R. Hughes}, {\em Mathematical foundations of
  elasticity}, Dover Publications, Inc., New York, 1994.

\bibitem{mclean2000strongly}
{\sc W.~C.~H. McLean}, {\em Strongly elliptic systems and boundary integral
  equations}, Cambridge University Press, 2000.

\bibitem{Polizzotto}
{\sc C.~Polizzotto}, {\em Nonlocal elasticity and related variational
  principles}, Internat. J. Solids Structures, 38 (2001), pp.~7359--7380.

\bibitem{Ponce16}
{\sc A.~C. Ponce}, {\em Elliptic {PDE}s, measures and capacities}, vol.~23 of
  EMS Tracts in Mathematics, European Mathematical Society (EMS), Z\"{u}rich,
  2016.

\bibitem{Romano}
{\sc G.~Romano, R.~Luciano, R.~Barretta, and M.~Diaco}, {\em Nonlocal integral
  elasticity in nanostructures, mixtures, boundary effects and limit
  behaviours}, Contin. Mech. Thermodyn., 30 (2018), pp.~641--655.

\bibitem{Scott23}
{\sc J.~M. Scott}, {\em The fractional {L}am\'e--{N}avier operator:
  Appearances, properties and applications}.
\newblock ArXiv preprint 2204.12029, 2022.

\bibitem{ShSp15}
{\sc T.-T. Shieh and D.~E. Spector}, {\em On a new class of fractional partial
  differential equations}, Adv. Calc. Var., 8 (2015), pp.~321--336.

\bibitem{ShSp17}
\leavevmode\vrule height 2pt depth -1.6pt width 23pt, {\em On a new class of
  fractional partial differential equations {II}}, Adv. Calc. Var.,  (2017).

\bibitem{Silling2000}
{\sc S.~A. Silling}, {\em Reformulation of elasticity theory for
  discontinuities and long-range forces}, J. Mech. Phys. Solids, 48 (2000),
  pp.~175--209.

\bibitem{Silhavy20}
{\sc M.~\v{S}ilhav\'{y}}, {\em Fractional vector analysis based on invariance
  requirements (critique of coordinate approaches)}, Contin. Mech. Thermodyn.,
  32 (2020), pp.~207--228.

\bibitem{Silhavy22}
\leavevmode\vrule height 2pt depth -1.6pt width 23pt, {\em Fractional strain
  tensor and fractional elasticity}, J. Elast.,  (2022).

\end{thebibliography}
%

\end{document}